\documentclass[11pt]{article}
\usepackage[margin=1in]{geometry}
\usepackage[utf8]{inputenc}

\usepackage{amsmath}
\usepackage{amssymb}
\usepackage{amsthm}
\usepackage{dsfont}
\usepackage{mathrsfs}
\usepackage{stmaryrd}
\usepackage[all]{xy}
\usepackage[mathcal]{eucal}
\usepackage{verbatim}  %
\usepackage{fullpage}  %
\usepackage{hyperref}
\usepackage{physics}

\usepackage[doi=false,isbn=false,maxcitenames=5,maxnames=4,firstinits=true]{biblatex}
\usepackage{breakurl}

\newcommand{\bbF}{\mathbb{F}}

\newcommand{\bbQ}{\mathbb{Q}}

\newcommand{\bbZ}{\mathbb{Z}}

\renewcommand{\phi}{\varphi}

\newtheorem{theorem}{Theorem}[section]
\newtheorem{proposition}[theorem]{Proposition}
\newtheorem{lemma}[theorem]{Lemma}
\newtheorem*{lemma*}{Lemma}
\newtheorem{corollary}[theorem]{Corollary}
\newtheorem{conjecture}[theorem]{Conjecture}

\theoremstyle{definition}
\newtheorem{definition}[theorem]{Definition}

\newtheorem*{remark}{Remark}
\newtheorem*{remarks}{Remarks}

\numberwithin{equation}{section}

\DeclareMathOperator{\GL}{GL}

\DeclareMathOperator{\ord}{ord}
\newcommand{\legendre}[2]{\genfrac{(}{)}{}{}{#1}{#2}}

\newcommand{\Addresses}{{%
  \bigskip
  \footnotesize

  \noindent S.~Dembner, \textsc{Department of Mathematics, University of Chicago, Chicago, IL 60637}\par\nopagebreak
  \textit{Email address}: \texttt{sdembner@uchicago.edu}

  \medskip

  \noindent V.~Jain, \textsc{Department of Mathematics, Massachusetts Institute of Technology, Cambridge, MA 02139}\par\nopagebreak
  \textit{Email address}: \texttt{vanshika@mit.edu}

}}

\makeatletter
\def\keywords{\xdef\@thefnmark{}\@footnotetext}
\makeatother

\begin{document}

\pagestyle{plain}

\author{Spencer Dembner, Vanshika Jain}
\title{Hyperelliptic Curves and Newform Coefficients}
\date{August 4, 2020}
\maketitle

\begin{abstract}

We study which integers are admissible as Fourier coefficients of even integer weight newforms. In the specific case of the tau-function, we show that for all odd primes $\ell < 100$ and all integers $m \geq 1$, we have
$$ \tau(n) \neq \pm \ell, \pm 5^m. $$
For general newforms $f$ with even integer weight $2k$ and integer coefficients, we prove for most integers $j$ dividing $2k-1$ and all ordinary primes $p$ that $a_f(p^2)$ is never a $j$-th power. We prove a similar result for $a_f(p^4)$, conditional on the Frey-Mazur Conjecture. Our primary method involves relating questions about values of newforms to the existence of perfect powers in certain binary recurrence sequences, and makes use of bounds from the theory of linear forms in logarithms. The method extends without difficulty to a large family of Lebesgue-Nagell equations with fixed exponent. To prove results about general newforms, we also make use of the modular method and Ribet's level-lowering theorem.

\end{abstract}


\section{Introduction and Statement of Results}

The Ramanujan tau-function $\tau(n)$ gives the coefficients of the weight $12$ modular form
\[\Delta(z) = \sum_{n=1}^\infty \tau(n) q^n := q \prod_{n=1}^\infty (1 - q^n)^{24} = q - 24q^2 + 252q^3 - 1472q^4 + \dots,\]
where $q = e^{2 \pi i z}$. In 1947, Lehmer posed the conjecture, still unproven, that $\tau(n) \neq 0$ for any $n \geq 1$. More recent work has focused on the variant problem of showing that $\tau(n)$ never takes certain odd values. In 1987, Murty, Murty and Shorey showed in \cite{murty1987odd} that for any odd $\alpha$, $\tau(n) = \alpha$ holds for at most finitely many $n$. However, the bounds in their result were far too large to be effective. Recently, Balakrishnan, Craig, Ono, and Tsai (\cite{balakrishnan2020short}, \cite{balakrishnan2020variants}) used a different method based on primitive prime divisors in Lucas sequences to rule out several specific odd primes as values of $\tau(n)$. Among their results is that for $n >1$, we have
$$ \tau(n) \notin \{\pm 1, \pm 3, \pm 5, \pm 7, \pm 13, \pm 17, -19, \pm 23, \pm 37, \pm 691\}.$$
Conditional on the Generalized Riemann Hypothesis, they also show that
$$\tau(n) \notin \left\{ \pm \ell: 41 \leq \ell \leq 97 \text{ with } \legendre{\ell}{5} = -1 \right\} \cup \{-11, -29, -31, -41, -59, -61, -71, -79, -89 \}. $$
Conspicuous in the statement of this result is the absence of the value $19$. In this paper, we resolve the question of whether $\tau(n)$ can equal $19$, and prove the following unconditional generalization:

\begin{theorem}\label{Maintheorem1}
If $\ell$ is any odd prime less than $100$, then $\tau(n) \neq \pm \ell$ for any $n$.
\end{theorem}

Theorem \ref{Maintheorem1} applies to finitely many odd primes. However, our techniques also yield the following result, which rules out infinitely many integers as possible values of the tau-function.

\begin{theorem}\label{Maintheorem2.a}
For all $m \geq 1$ and all $n$, we have $\tau(n) \neq \pm 5^m$.
\end{theorem}

Beyond determining which values are admissible, we can limit the possible factorizations of the values $\tau(p^4)$.

\begin{theorem}\label{Maintheorem2.b}
For all primes $p$, $\tau(p^4)$ is divisible only by $p$, or by primes $\ell \equiv 0, 1, 4 \pmod 5$.
\end{theorem}

Our final theorems are more general, applying to all newforms of sufficiently large even weight $2k$ with integer coefficients. Among other things, we show that certain Fourier coefficients are never $\ell$-th powers, for sufficiently large primes $\ell$ dividing $2k-1$.

\begin{theorem}\label{Maintheorem3}
If $f(z) = \sum a_f(n) q^n$ is a newform with integer coefficients and even integer weight $2k$, then the following are true.
\begin{enumerate}
    \item If $p$ is an ordinary prime for $f(z)$, i.e., $p \nmid a_f(p)$, then $a_f(p^2) \neq m^j$ for any $j \geq 4$ dividing $2k-1$ and any nonzero integer $m$. In particular, if $2k \geq 6$, then $a_f(p^2) \neq m^{2k-1}$.
    \item Assume the Frey-Mazur Conjecture. If $2k-1$ is divisible by a prime $\ell \geq 19 $, and $p \neq 2,5$ is ordinary, then $a_f(p^4) \neq m^\ell$ for any nonzero integer $m$.
    \item Assume the Frey-Mazur Conjecture. Suppose that $f$ has residually reducible mod $2$ Galois representation, and level $N$ coprime to $19$. Suppose that $2k-1$ is divisible by a prime $\ell > 19$. If $a_f(5^4), a_f(19^4) \neq \pm 19$, then $a_f(n) \neq \pm 19$ for any $n$.
\end{enumerate}
\end{theorem}

Part (1) and (2) of Theorem \ref{Maintheorem3} are analogous to Theorem 1.4 of \cite{balakrishnan2020variants}, which says that for forms of sufficiently high weight, a given value $\pm \ell^m$ cannot occur as a Fourier coefficient. In Theorem \ref{thm: Neglected modularity result}, we prove another result of a similar flavor.

In Subsection \ref{subsec: Newform preliminaries}, we state several basic facts about newforms, and prove Theorem \ref{Maintheorem2.b}. As we discuss in this subsection, the main difficulty in ruling out possible values of the tau-function is classifying points on certain hyperelliptic curves and Thue equations. In Sections \ref{Preliminariessection}, \ref{ElementaryResultsSection}, and \ref{sec: proof of theorem 1.1}, we focus on equations of the form
$$ Y^2 = 5X^{22} + 4 \alpha, $$
which arise in proving that $\tau(p^4) \neq \alpha$. These equations are closely related to the arithmetic of real quadratic fields, and by extension, to the existence of perfect $11$-th powers in certain binary recurrence sequences. Specifically, the sequences in question are of \emph{Fibonacci type}: they all satisfy the same recurrence
$$ a_{n+2} = a_{n+1} + a_n, $$
as the Fibonacci sequence $(u_n)$ and the classical Lucas numbers $(v_n)$. Due to work of Bugeaud, Mignotte, and Siksek \cite{bugeaud2006classical}, all perfect powers in the sequences $(u_n), (v_n)$ are known. However, less is known about perfect powers in arbitrary sequences of Fibonacci type. In Section \ref{ElementaryResultsSection}, we use purely elementary congruence arguments to rule out perfect $11$-th powers in many of the sequences relevant to values of the tau-function. We also deduce Theorem \ref{Maintheorem2.a}, and we show that
$$ a_f(p^4) \notin \{ \pm 31, \pm 59, \pm 61, \pm 79, \pm 101, \pm 139, \pm 149, \pm 151, \pm 179, \pm 191, \pm 199, \pm 19^2 \},  $$
for newforms $f$ with integer coefficients, weight $12,14,18,$ or $20$, and any level coprime to $p$.

In some cases, congruences are insufficient to rule out perfect $11$-th powers, because $\pm 1$ appears either in the sequence or in its continuation to negative values. In these cases, although congruences alone cannot rule out nontrivial perfect powers, they can be used to show that any perfect power $a u_n + b v_n$ must have very large index $n$. In Section \ref{sec: proof of theorem 1.1}, we combine this observation with upper bounds coming from the theory of linear forms in logarithms to show that the remaining sequences relevant to Theorem \ref{Maintheorem1} likewise have no nontrivial perfect powers. The method we use to solve these equations via perfect powers in recurrences generalizes without difficulty to any Lebesgue-Nagell equation of the form
$$ x^2 +D = C y^{2n}, $$
where $n$ is fixed and $\bbQ(\sqrt{C})$ has class number one.

In the rest of Section \ref{sec: proof of theorem 1.1}, we complete the proof of Theorem \ref{Maintheorem1} by solving the relevant Thue equations. All these equations were computed in \cite{balakrishnan2020variants}, but these computations were often conditional on the Generalized Riemann Hypothesis. By using an upper bound proven in \cite{bilu2001existence}, we are able to perform the same computations unconditionally. 

In Sections \ref{sec:modularapproaches} and \ref{sec: Applications of modularity}, we give the proof of Theorem \ref{Maintheorem3}. Our work in these sections is based on the modularity theorem proved by Wiles and others (\cite{Wiles1995modular},\cite{taylor1995heckealgebras},\cite{breuil2001modularity}), as well as Ribet's level-lowering theorem \cite{ribet1991levellowering}. Most but not all of the results we prove are conditional on the unproven Frey-Mazur Conjecture, which says that for large enough primes $p$, an elliptic curve $E / \bbQ$ is determined up to isogeny by its $p$-torsion Galois representation $E[p]$ (see Section \ref{sec:modularapproaches} for a precise statement).

\subsection*{Acknowledgments}

We thank Professor Ken Ono, Wei-Lun Tsai, and William Craig for their suggestions and advice on this project. Both authors were supported by the NSF (DMS-2002265), the NSA (H98230-20-1-0012), the Templeton World Charity Foundation, and the Thomas Jefferson Fund at the University of Virginia.

\section{Preliminaries}\label{Preliminariessection}

In the first subsection, we state several basic results about newform coefficients. Then, in the second subsection, we explain how certain coefficients are related to perfect powers in recurrences.

\subsection{Newform Coefficients} \label{subsec: Newform preliminaries}
We begin by stating several key properties of newforms. For background on the theory of modular forms, see \cite{apostol2012modular}, \cite{Cohen2017modular}, and \cite{Ono2004web}. In this section, $f$ denotes a newform with integer coefficients, even weight $2k$, and level $N$. We write the Fourier expansion of $f$ as
\begin{equation}
   f(z) =  q + \sum_{n \geq 2} a_f(n) q^n.
\end{equation}

\begin{theorem}\label{thm:propertiesofnewforms}
The following are true.

\begin{enumerate}
    \item If $\gcd(n_1, n_2) = 1$, then $a_f(n_1 n_2) = a_f(n_1) a_f(n_2)$.
    \item If $p \nmid N$ is prime and $m \geq 2$, then
    $$ a_f(p^m) = a_f(p) a_f(p^{m-1}) - p^{2k-1} a_f(p^{m-2}). $$
    \item If $p \mid N$ is prime, then we have
    $$ a_f(p^m) = \begin{cases} (\pm 1)^m p^{(k-1)m} & \text{ if } \ord_p(N) = 1, \\ 0 & \text{ if } \ord_p(N) \geq 2. \end{cases} $$
\end{enumerate}
\end{theorem}

\noindent By repeatedly applying the recurrence from part (2) of Theorem \ref{thm:propertiesofnewforms}, one can show that primes $p$ for which $a_f(p^m) = \pm \alpha$ give rise to solutions to Diophantine equations.

\begin{corollary}[\cite{balakrishnan2020variants}, Lemma 5.1]
With notation as in Theorem \ref{thm:propertiesofnewforms}, suppose that $p \nmid N$ is prime, and let $\alpha$ be any integer. Then we have the following:

\begin{enumerate}\label{cor: recurrencesgivediophantineequations}
    \item If $a_f(p^2) = \alpha$, then $(p, a_f(p))$ is an integer point on
    \begin{equation} \label{C Curves} 
        C_{k, \alpha}: Y^2 = X^{2k-1} + \alpha. 
    \end{equation}
    \item If $a_f(p^4) = \alpha$, then $(p, 2 a_f(p)^2 - 3 p^{2k-1})$ is an integer point on
    \begin{equation} \label{themainhyperellipticcurves}
        H_{2k-1,\alpha}: Y^2 = 5 X^{2(2k-1)} + 4 \alpha.
    \end{equation}
    
    \item For $m$ a positive integer, let $F_{2m}(X,Y)$ be the homogeneous polynomial defined by the condition
    $$ \frac{1}{1-\sqrt{Y} T + X T^2} = \sum_{i=0}^{\infty} F_{i}(X,Y) \cdot T^i. $$
    Then for all $m \geq 1$, $(p^{2k-1}, a_f(p)^2)$ is a solution to the \emph{Thue equation}
    $$ F_{2m}(X,Y) = a_f(p^{2m}). $$
\end{enumerate}

\end{corollary}

For later use, we record the following alternate expression for $F_{2m}(X,Y)$:

\begin{equation} \label{eq:Thueproductformula} 
F_{2m}(X,Y) = \prod_{k=1}^m \left( Y - 4 X \cos^2 \left( \frac{\pi k}{2m +1} \right) \right). 
\end{equation}

Another consequence of Theorem \ref{thm:propertiesofnewforms} is that for $p \nmid N$, the values $a_f(p^m)$ form a \emph{Lucas sequence.} In the paper \cite{balakrishnan2020variants}, the authors use this observation, along with deep work of Bilu, Hanrot, and Voutier \cite{bilu2001existence} on prime divisors in Lucas sequences, to rule out specific odd prime values as possible coefficients of certain newforms. They prove the following key result, which also admits a partial generalization to weight $4$.

\begin{theorem}[\cite{balakrishnan2020variants}, Theorem 3.2]\label{thm:divisibilityfortau}
Let $f$ have weight $2k \geq 6$ and residually reducible mod $2$ Galois representation. If $\ell$ is an odd prime not dividing $N$ such that $a_f(n) = \pm \ell^m$ for some $m \geq 0$, then $n = p^{d-1}$, where $p \nmid N$ and $d$ are odd primes, and $d \mid \ell (\ell^2-1)$.
\end{theorem}

Thus, in order to rule out the possibility that $a_f(n) = \pm \ell^m$ for a fixed odd prime $\ell$, it is enough to show that $a_f(p^{d-1}) \neq \pm \ell^m$, where $d$ is one of finitely many odd primes dividing $\ell(\ell^2-1)$. For $d \geq 7$, solutions to $a_f(p^{d-1}) = \pm \ell^m$ give solutions to Thue equations, which have been intensively studied. For $d = 3$, we obtain solutions to hyperelliptic curves of the form \eqref{C Curves}. For $|\alpha| \leq 100$, equation \eqref{C Curves} has been completely solved by work of Bugeaud, Mignotte, and Siksek \cite{bugeaud2006classical2} and Barros \cite{barros2010lebesgue}.

This means that for small primes $\ell$, the main obstacle is the case $d = 5$, where we obtain a solution to \eqref{themainhyperellipticcurves}. This case is closely related to the arithmetic of the real quadratic field 
$$K:= \bbQ(\sqrt{5}).$$
The following elementary proposition shows that \eqref{themainhyperellipticcurves} will have few solutions if $\alpha$ is divisible by any prime equivalent to $2$ or $3$ mod $5$. 

\begin{proposition}\label{prop: alpha can't have divisors 2,3 mod 5}
Fix some $\alpha$ and some $d \geq 1$, and suppose that $(x,y)$ is an integer point on one of the curves $Y^2 = 5 X^{2d} \pm 4 \alpha$. Then any prime $\ell \equiv 2, 3 \pmod 5$ which divides $\alpha$ must divide $x$ as well.
\end{proposition}

\begin{proof}
Suppose that some prime $\ell \neq 5$ divides $\alpha$. Then we have
$$ y^2 \equiv 5 x^{2d} \pmod{\ell}. $$
If $x \not \equiv 0 \pmod \ell$, then $x^{2d}$ is a nonzero square mod $\ell$. Since $5 x^{2d}$ is a nonzero square as well, it follows that $5$ is a square mod $\ell$. By quadratic reciprocity, this implies that
$$ \legendre{5}{\ell} = \legendre{\ell}{5} = 1. $$
Therefore, either $\ell \equiv 1,4 \pmod 5$, or $\ell$ divides $x$. \qedhere
\end{proof}

By Theorem \ref{thm:divisibilityfortau}, we know that $\tau(p^4) \neq \pm \ell^m$ for primes $\ell \equiv \pm 2 \pmod 5$. As a consequence of Proposition \ref{prop: alpha can't have divisors 2,3 mod 5}, we show the following stronger result.

\begin{corollary}\label{cor:tau p^4 not divisible by primes 2 3 mod 5}
The following are true.
\begin{enumerate}
    \item For any prime $p$, $\tau(p^4)$ can only be divisible by $p$, or by primes $\ell \equiv 0,1, 4 \pmod 5$. In particular, if $\tau(p^4)$ is divisible by any prime $2, 3$ mod $5$, then $p$ is not ordinary for the $\tau$-function, i.e., $p$ divides $\tau(p)$.
    \item Let $f$ be an even weight newform with integer coefficients, and suppose $p \nmid N$. Then $a_f(p^4)$ can only be divisible by $p$, or by primes $\ell \equiv 0,1, 4 \pmod 5$.
\end{enumerate}
\end{corollary}

\begin{proof}[Proof of Corollary \ref{cor:tau p^4 not divisible by primes 2 3 mod 5} and Theorem \ref{Maintheorem2.b}]

To prove both statements, suppose that $p \nmid N$ and that $a_f(p^4) = \alpha$. Then we obtain a point $(p, 2 a_f(p)^2 - 3 p^{2k-1})$ on the curve $Y^2 = 5 X^{2(2k-1)} + 4 \alpha$. It follows that any prime dividing $\alpha = a_f(p^4)$ must either be $0,1,4 \pmod 5$, or else must divide $p$ as well. In this case, the prime must be $p$. \qedhere
\end{proof}

\subsection{Recurrences of Fibonacci Type} \label{subsec:recurrencesoffibtype}

To analyze solutions to \eqref{themainhyperellipticcurves} when $\alpha$ is arbitrary, we first establish a connection to powers in certain recurrences. Let $N = \operatorname{Norm}_{K / \bbQ} \colon K \to \bbQ$ denote the norm map, defined by 
$$N(a+b\sqrt{5}) = (a+b\sqrt{5})(a-b\sqrt{5}) = a^2 - 5 b^2.$$
If $(X, Y)$ is an integer point on \eqref{themainhyperellipticcurves} for any integer $\alpha$, then we have 
\begin{equation} \label{eq:modifiedpell}
N(Y+ X^{2k-1} \sqrt{5}) = \pm 4 \alpha. 
\end{equation}

This is closely related to the Pell equation $N(a + b\sqrt{5}) = \pm 1$. Let $O_K$ denote the ring of integers of $K$. The elements $\zeta \in O_K$ such that $N(\zeta) = \pm 1$ are exactly the units of $O_K$. Since $K$ is a real quadratic field, by Dirichlet's unit theorem, the unit group $O_K^{\times}$ has rank $1$, and in this case, a fundamental unit is given by 
$$\omega := \frac{1+\sqrt{5}}{2}.$$
In order to study \eqref{eq:modifiedpell}, we begin with the observation that up to multiplication by units, there are only finitely many elements of norm $\pm 4 \alpha$. Suppose that $(\zeta_1),\ldots, (\zeta_r)$ are all the ideals of norm $4 |\alpha|$, and write $\zeta_i = 2 (a_i + b_i \sqrt{5})$. Then for some $i$ and some $k$, we have
$$ Y+ X^{2k-1} \sqrt{5} =  \pm 2 (a_i + b_i \sqrt{5}) \omega^k. $$
For any $\gamma = a +b \sqrt{5} \in K$, we write its conjugate as $\overline{\gamma} = a - b \sqrt{5}$. Summing conjugates, we see that
\begin{equation} \label{eq:solforXintermsofsequences}
    X^{2k-1} = \pm \left( \frac{a_i}{\sqrt{5}}(\omega^k - \overline{\omega}^k) + b_i(\omega^k + \overline{\omega}^k) \right).
\end{equation}

\begin{definition}
For $n \geq 0$, the Fibonacci numbers $u_n$ and Lucas numbers $v_n$ are defined by
$$ u_n = \frac{\omega^n - \overline{\omega}^n}{\omega - \overline{\omega}}, \quad v_n = \omega^n + \overline{\omega}^n. $$
\end{definition}
We can continue the Fibonacci sequence to negative values of $j$ by requiring that it satisfy the same recurrence; for example, $u_{-1} = 1$, and $u_{-2} = -1$. By induction, if $j < 0$ we have
$$ u_j = (-1)^{j+1} u_{-j}. $$
Similarly, we get the same identity for $v_n$. Therefore, in \eqref{eq:solforXintermsofsequences}, we can assume $k \geq 0$. 

\begin{definition}
We call $(x_n)$ a \emph{sequence of Fibonacci type} if $x_n = a u_n + b v_n$ for some $a, b \in \bbZ$.  
\end{definition}

Our work so far shows that an integer solution to \eqref{themainhyperellipticcurves} gives rise to a perfect power in one of finitely many sequences of Fibonacci type. Since $-1$ is a perfect $(2k-1)$-st power for all $k$, we can ignore the sign in the statement of the following proposition.

\begin{proposition}\label{prop:pellsolutionsgiveperfect powers}
Suppose that $(X,Y)$ satisfies \eqref{eq:modifiedpell}. With notation as in the previous section, there is some $i \leq r$ such that $X^{2k-1} = \pm ( a_i u_n + b_i v_n)$.
\end{proposition}
Conversely, given a perfect power in some sequence $( a u_n + b v_n )$, we obtain a corresponding point on the hyperelliptic curve.

\begin{lemma}\label{lem:perfectpowersgivecurvepoints}
Suppose that for some $d \geq 1$ and some $x$, we have $a u_n + b v_n = x^d$. Then for some $y$, $(x,y)$ is an integer point on
$$ Y^2 = 5 X^{2d} \pm 4 (a^2 - 5 b^2). $$
\end{lemma}

\begin{proof}
We have already seen that for any $k \geq 0$, we have
$$ 2 (a + b \sqrt{5}) \left(\frac{1+\sqrt{5}}{2}\right)^k = Y + (a u_k + b v_k) \sqrt{5}, $$
for some integer $Y$. Setting $k = n$ and taking norms on both sides, we have
$$ 4 (a^2 - 5 b^2) \cdot (-1)^n = Y^2 - 5 X^{2d}. $$
We can rewrite this as
$$ Y^2 = 5 X^{2d} + (-1)^n \cdot 4(a^2 - 5 b^2),  $$
and this gives the desired conclusion. \qedhere
\end{proof}

In 1982 and 1983, Peth\H{o} \cite{Petho1982perfect} and Shorey and Stewart \cite{shorey1983diophantine} independently proved that any binary recurrence sequence has finitely many perfect powers. More recently, Silliman and Vogt \cite{silliman2015powers} applied the Frey-Mazur Conjecture on Galois representations associated to elliptic curves to outline a conditional classification of perfect powers in Lucas sequences satisfying recurrences other than the Fibonacci recurrence (for the statement of the Frey-Mazur Conjecture, see Conjecture \ref{conj:Frey-Mazur}). However, aside from the Fibonacci and Lucas sequences (\cite{bugeaud2006classical}), there has been little work on perfect powers in sequences of Fibonacci type.

\section{Elementary Results} \label{ElementaryResultsSection}

In this section, we prove several results about perfect powers and newform coefficients which depend only on the elementary considerations stated so far. In the next section, we will use upper bounds on solutions to Thue equations to solve the remaining cases.

Given a sequence $(a u_n + b v_n)$ of Fibonacci type, we can often rule out the possibility of perfect $d$-th powers, for specific values of $d$, by an elementary computation. The key idea for these computations is that, by the pigeonhole principle, any linear recurrence is periodic mod $m$ for any $m$. In fact, all the sequences we are interested in will have the same period mod $m$, except for finitely many $m$. Let $\pi(m)$ denote the period of the Fibonacci sequence mod $m$.

\begin{lemma}
Let $(x_n) = (a u_n + b v_n)$ be a sequence of Fibonacci type. If $a^2 - 5 b^2$ is coprime to $m$, then $(x_n)$ has period $\pi(m)$ mod $m$.
\end{lemma}

\begin{proof}
Let $A$ be the matrix
$$ \begin{pmatrix} 1 & 1 \\ 1 & 0 \end{pmatrix}. $$
Then we have 
$$ A (x_{n+1}, x_n) = (x_{n+2}, x_{n+1}). $$
Since $\det(A) = -1 $, $A$ is an element of $\GL_2(\bbZ / m \bbZ)$. If $A$ has order $k$ in $\GL_2(\bbZ  /m \bbZ)$, then it is clear that the period of $(x_n)$ divides $k$. Conversely, suppose that $A^k \neq I \in \GL_2(\bbZ / m \bbZ)$. The matrix with vectors $(x_0, x_1) = (2b, a+b)$ and $(x_1, x_2) = (a+b, a+3b)$ has determinant $5 b^2 - a^2$, which is coprime to $m$. It follows that the two vectors form a basis of $(\bbZ / m \bbZ)^2$. Therefore, either $x_k \not \equiv x_0 \pmod m$, or $x_{k+1} \not \equiv x_1 \pmod m$. Thus, we see that the period is given by the order of $A$. The case $a = 1, b = 0$ corresponds to the Fibonacci sequence, with period $\pi(m)$. Since $1$ is a unit mod $m$ for all $m > 1$, it follows that the period of $(x_n)$ is equal to $\pi(m)$. \qedhere

\end{proof}

Fix a prime $d$. To rule out the possibility of $d$-th powers in a sequence $(a u_n + b v_n)$, we can fix a prime $q \equiv 1 \pmod d$ which is coprime to $a^2 - 5 b^2$ and study the indices $k$ for which $a u_k + b v_k$ is a perfect $d$-th power mod $q$. If $q^m$ divides $\pi(q)$ for some prime $q$ and some $m \geq 1$, then we obtain information about the possible congruence classes of $k$ mod $q^m$. By comparing the congruences obtained from multiple primes, we can often show directly that some sequences of Fibonacci type contain no perfect $d$-th powers. 

\begin{proposition}\label{prop:congruencechecking}
Consider some sequence $(a u_n + b v_n)$ of Fibonacci type. Suppose that for some $m \geq 2$ coprime to $a^2 -5 b^2$, and for all $k < \pi(m)$, $ a u_k + b v_k$ is not a perfect $d$-th power mod $m$. Then $a u_n + b v_n$ is never a perfect $d$-th power. 
\end{proposition}

\begin{theorem}\label{thm:rulingoutperfectpowersbycongruences}

Consider the powers $d = 3,5,7, 11, 13, 17,$ and $19$. The sequence $(a u_n + b v_n)$ contains no perfect $d$-th powers for the following values of $(a,b)$:
\begin{gather*}
     \{ (1, \pm 4), (7, \pm 4), (21, \pm 4), (9 \pm 2),
     (8, \pm 1), (12, \pm 7), \\ (12, \pm 1), (13, \pm 2), (13, \pm 8), (1, \pm 6), (14, \pm 1), (18, \pm 5) \},\\
\end{gather*}
except for
$$\begin{cases}
(1, \pm 4), (7, \pm 4), (21, \pm 4) & d = 3,\\
(14, \pm 1) & d = 5, \\
(7, \pm 4) & d = 7.
\end{cases}$$
Additionally, for any $p \neq 2$, any $ d \geq 3$, and any $m \geq 1$, the sequences $p^m u_n$ and $p^m v_n$ have no perfect prime powers $q^d$, unless $q = p$ and $d = m+1$.
\end{theorem}

\begin{proof}[Proof of Theorem \ref{thm:rulingoutperfectpowersbycongruences}]
By the main result from \cite{bugeaud2006classical}, the sequences $u_n$ and $v_n$ have no perfect powers of the form $p^m$, where $m \geq 2$ and $p \neq 2$ is an odd prime. Suppose that $p^m u_n = q^d$ for some $d \geq 3$ and some prime $q$. Then we have $ q = p$, so $u_n = p^{d-m}$, which is impossible unless $d = m+1$.

For the remaining cases, we use Proposition \ref{prop:congruencechecking}. The key idea is that when an element $a u_k + b v_k$ is a power modulo some prime $p$, then this gives us congruence information about the index $k$. These congruences hold modulo the prime powers dividing $\pi(p)$, rather than modulo $p$. 

We implemented an algorithm in Python (see \cite{DembnerJain2021code} for all code used in the paper) which does the following: for all sufficiently small primes $p \equiv 1 \pmod d$ coprime to $a^2 - 5 b^2$, it finds the indices $k < \pi(p)$ such that $a u_k + b v_k$ is a perfect $d$-th power mod $p$. Then, for all prime powers $\ell^m$ dividing $\pi(p)$, it throws away the indices $k$ which do not satisfy the congruences mod $\ell^m$ already obtained from previous primes. If $a u_n + b v_n$ is a perfect $d$-th power for some $n$, then for all primes powers $\ell^m$ dividing $\pi(p)$, $n$ must be congruent mod $\ell^m$ to one of the remaining indices $k < \pi(p)$. In the specified cases, the program shows that the congruences on the index $n$ obtained from primes $p < 10,000$ are not satisfiable, which rules out the possibility that $ a u_n  +b v_n$ contains any perfect $d$-th power. \qedhere

\end{proof}

\begin{remarks}
\hfill
\begin{enumerate}
    \item The sequences and powers were chosen for their relevance to newform coefficients; the method should work equally well for most powers $d$.
    \item For some sequences, the values for small $n$ make it impossible to rule out the existence of perfect powers in this way. For example, the Fibonacci sequence $( u_n )$ has several small perfect power values, which means larger perfect powers cannot be ruled out by congruences alone.
    \item The sequences $u_n + 2 v_n$, which are related to the question of whether $\tau(n) = 19$, cannot be handled for a similar reason: extending indices backward, we see that $u_{-1} + 2 v_{-1} = -1$ is a perfect power. The same obstruction occurs for any sequence $( a u_n + (a+1) v_n )$. 
\end{enumerate}
\end{remarks}

By using the results obtained above about perfect powers in sequences $a u_n + b v_n$, we can rule out many specific powers of primes $\ell \equiv \pm 1 \pmod 5$ as possible values of $a_f(p^4)$.

\begin{corollary}\label{cor:rulingoutcoeffpm1}
We have
$$ a_f(p^4) \notin \{ \pm 31, \pm 59, \pm 61, \pm 79, \pm 101, \pm 139, \pm 149, \pm 151, \pm 179, \pm 191, \pm 199, \pm 19^2 \}, $$
where $f$ is any newform with integer coefficients and weight $2k = 12, 14, 18,$ or $20$.
\end{corollary}

\begin{proof}

Let $\alpha$ be one of the values in the statement. If $p \mid N$, then by Theorem \ref{thm:propertiesofnewforms} part (3), $a_f(p^4)$ is either $0$ or $\pm p^j$, where $j > 2$. Therefore, we may assume that $p \nmid N$. It follows by Corollary \ref{cor: recurrencesgivediophantineequations} that we obtain an integer point $(p, 2 a_f(p)^2 -3 p^{2k-1})$ on $H_{2k-1, \alpha}$.

First, we give the proof for $\alpha = 31$, which is nearly identical to the proof for the other listed primes $\ell \equiv \pm 1 \pmod 5$. The prime $31$ splits in $\bbQ(\sqrt{5})$ in the following way:
$$ (31) = (7+4 \sqrt{5})(7 - 4 \sqrt{5}). $$
We have already seen that if $a_f(p^4) =  \pm 31$, then this gives rise to an integer point on $Y^2 = 5 X^{2(2k-1)} \pm 4 \cdot 31$. By Proposition \ref{prop:pellsolutionsgiveperfect powers}, this gives rise to a perfect $(2k-1)$-st power in either the sequence $( 7 u_n + 4 v_n )$, or the sequence $( 7 u_n - 4 v_n )$. By Theorem \ref{thm:rulingoutperfectpowersbycongruences}, neither sequence contains such a perfect power when $2k = 12, 14, 18$, or $20$. 

Now, we give the proof for $\alpha = 19^2$. Up to multiplication by units, the three elements of $O_K$ with norm $\pm 19$ are $19$, $21 + 4 \sqrt{5}$, and $21 - 4 \sqrt{5}$. As in the previous case, if $a_f(p^4) = \pm 19^2$, then this gives rise to a perfect power $X^{2k-1}$ in one of the sequences $19 u_n, 21 u_n + 4 v_n$, or $21 u_n - 4 v_n$. Furthermore, the perfect power $X^{2k-1} = p^{2k-1}$ is a power of an odd prime. By Theorem \ref{thm:rulingoutperfectpowersbycongruences}, the sequences $( 21 u_n \pm 4 v_n) $ contain no perfect $(2k-1)$-st powers, and the sequence $(19 u_n )$ contains no powers of the form $p^{2k-1}$, where $p$ is an odd prime. \qedhere

\end{proof}

\begin{remark}
Similar results hold for newforms of weight $4,6,$ or $8$, except for values corresponding to the exceptional cases of Theorem \ref{thm:rulingoutperfectpowersbycongruences}.
\end{remark}

We have already seen that due to the work of Bugeaud, Mignotte, and Siksek, the sequences $5^m u_n$ and $5^m v_n$ contain no perfect prime powers except for $5$. Along with a congruence due to Ramanujan, this allows us to show $\tau(n)$ is never a power of $5$.

\begin{theorem}[Ramanujan]\label{thm:congruence for 5}
For all $n$, we have
$$ \tau(n) \equiv n \sigma_1(n) \pmod{5}.$$
\end{theorem}

\begin{corollary}\label{cor: power of 5}
For all $n$, and all $m \geq 1$, we have $\tau(n) \neq \pm 5^m$.
\end{corollary}

\begin{proof}[Proof of Corollary \ref{cor: power of 5} and Theorem \ref{Maintheorem2.a}]
By Theorem \ref{thm:divisibilityfortau}, if $\tau(n) = \pm 5^m$, then we may assume that $n = p^{d-1}$ where $d = 3, 5$ and $p$ is an odd prime. By Theorem \ref{thm:congruence for 5}, it is not possible that $\tau(p^2) \equiv 0 \pmod 5$ for $p \neq 5$ prime, and $\tau(5^2)$ is not a power of $5$. Therefore, we only need to rule out the possibility that $\tau(p^4) = \pm 5^m$.

The solutions to $\tau(p^4) = \pm 5^m$ correspond to solutions $(p^{11}, Y)$ to $Y^2 = 5X^{2} \pm 4 \cdot 5^m$. Since $5$ ramifies in $\bbQ(\sqrt{5})$ and $2$ is inert, $2 \cdot \sqrt{5}^m$ is the only element of $O_K$ with norm $4 \cdot 5^{m}$. It follows that we obtain a perfect power of the form $p^{11}$ in either the sequence $(5^m u_n) $ or the sequence $(5^m v_n)$. By Theorem \ref{thm:rulingoutperfectpowersbycongruences}, neither sequence contains such a prime power, except if $p = 5$ and $m = 10$. Since $\tau(5^4)$ is not a power of $5$, it follows that $\tau(n) \neq \pm 5^m.$ \qedhere
\end{proof}

\section{Proof of Theorem \ref{Maintheorem1}}\label{sec: proof of theorem 1.1}

There are two steps remaining in the proof of Theorem \ref{Maintheorem1}. First, we need to rule out nontrivial perfect $11$-th powers in the sequences $(a u_n + b v_n)$ corresponding to primes $\ell < 100$ with $\ell \equiv \pm 1 \pmod 5$ which were not handled in the Theorem \ref{thm:rulingoutperfectpowersbycongruences}. Second, we need to unconditionally classify solutions to the Thue equations $F_{d-1}(X,Y) = \pm \ell$, where $d \mid \ell (\ell^2-1)$ is one of the odd prime values allowed by Theorem \ref{thm:divisibilityfortau}. 

We begin by handling the first task. By using a suitably refined version of the congruence arguments from Section \ref{ElementaryResultsSection}, a computer can quickly show that if a given sequence $ (a u_n + b v_n )$ has any nontrivial $11$-th powers, then the index $n$ must be very large. A sieve method based on a similar principle was described in \cite{bugeaud2006classical}, and was implemented in \texttt{SageMath} \cite{sagemath} by the authors of \cite{silliman2015powers}. Using this sieve, we were able to rule out perfect powers with small indices in all of the sequences relevant to primes we could not handle in Corollary \ref{cor:rulingoutcoeffpm1}.

\begin{proposition}\label{prop: 10^300 checking}
For the following values of $(a,b)$, if $n \leq 10^{300}$, then $a u_n + b v_n$ is not a perfect $11$-th power, except possibly $\pm 1$:
$$\{(1, \pm 2), (4, \pm 1), (4, \pm 3), (11, \pm 4), (14, \pm 5), (6, \pm 5 \}. $$
\end{proposition}

On its own, this method can never rule out all perfect $11$-th powers. However we can complete the argument by pairing these computations with upper bounds coming from the theory of linear forms in logarithms. The following key result appears in the work of Bugeaud-Mignotte-Siksek, Barros, and others. It says that solutions to $Y^2 = 5 X^{22} \pm 4 \ell$ give rise to points on certain Thue equations.

\begin{theorem}[\cite{barros2010lebesgue} Theorem 2.1, \cite{bugeaud2006classical2} Proposition 3.1] \label{thm:barrosthueequations}
Suppose that $(x, y)$ is a solution to $x^2 = C y^n - D$. Write $D = d q^2$, where $d$ is squarefree. Let $K = \bbQ(\sqrt{-d})$, and let $O_K$ denote its ring of integers. Let $\{1, \omega \}$ denote an integral basis of $O_K$, and let $\overline{\omega}$ denote the conjugate of $\omega$. Then there exists a finite set $\Gamma$ of pairs $\{\gamma_+, \gamma_- \}$ of elements from $K$, such that
$$ x = \frac{1}{2} (\gamma_+(A+B \omega)^n + \gamma_-(A+B \overline{\omega})^n), $$
where $(A,B)$ is a solution of the Thue equation
$$ 2 q = \frac{1}{\sqrt{-d}} (\gamma_+ (A +B \omega)^n - \gamma_-(A+B \overline{\omega})^n). $$
\end{theorem}

Thanks to the theory of linear forms in logarithms, we now know explicit bounds for the size of solutions to Thue equations. To state the result we will use, let $F(x,y) = b$ be a Thue equation. Let $M$ be the field $\bbQ(\alpha)$, where $\alpha$ is any root of the polynomial $F(X,1)$. Let $R = R_M$ denote the regulator of $M$, let $H$ be an upper bound for the absolute value of the coefficients of $F$, and let $B$ be an upper bound for the absolute value of $b$. Furthermore, let $n$ denote the degree of the extension $M / \bbQ$, and let $r$ denote the rank of the unit group $M^{\times}$.

\begin{theorem} [\cite{bugeaud1996bounds}, Theorem 3]\label{thm:bounds for thue equation}
All solutions $x, y$ of $F(x,y) = b$ satisfy
$$ \max \{ |x|, |y| \} < \exp \{ c_3 R (\log^* R) (R + \log(HB)) \}, $$
where $\log^*(x) = \max(\log(x), 1)$, and 
$$ c_3 = c_3(n,r) = 3^{r+27} (r+1)^{7r+19} n^{2n+6r+14}. $$
\end{theorem}

In order to use this result, we need a bound on the regulators of the number fields arising from our Thue equations. The following result supplies such a bound. 

\begin{lemma}[\cite{Bugeaud2008integral}, Lemma 5.1]\label{lem: Boundsonregulators}
Let $K$ be a number field with degree $d = u + 2v$ where $u$ and $v$ are respectively the numbers of real and complex embeddings. Denote the absolute discriminant by $D_K$, the regulator by $R$, and the number of roots of unity in $K$ by $w$. Suppose that $L$ is a real number such that $D_K \leq L$. Let 
$$a = 2^{-v} \pi^{-d/2} \sqrt{L}.$$
Define the function $f_K(L, s)$ by 
$$f_K(L, s) = 2^{-u} w a^s \left(\Gamma \left(\frac{s}{2}\right) \right)^u (\Gamma(s))^{v} s^{d+1} (s-1)^{1-d},$$
and let $B_K(L) = \min\{f_K(L, 2 - t/1000) \colon t = 0, 1, \dots, 999\}$. Then $R < B_K(L)$. 
\end{lemma}

\begin{corollary}
Let $K / \bbQ$ be a degree $11$ extension of $\bbQ$, and fix notation as in Lemma \ref{lem: Boundsonregulators}. Then we have
$$ R < 2^2 L. $$
\end{corollary}

\begin{proof}
Taking $t = 0$ in Lemma \ref{lem: Boundsonregulators}, we have
$$ R < 2^{-u} w a^2 1^u 1^v 2^{d+1} (1)^{1-d}  = 2^{-u} w 2^{-2v} \pi^{-d} L 2^{d+1}. $$
We have $u+2v = d$. Furthermore, we have $w = 2$, since there is no nontrivial cyclotomic extension of degree $11$. Therefore, the desired result follows. \qedhere

\end{proof}

\begin{theorem}\label{thm: Finishing the hard sequences}
Let $(a u_n + b v_n)$ be one of the sequences listed in Proposition \ref{prop: 10^300 checking}. Then $a u_n + b v_n$ is never a perfect $11$-th power, except for possibly $\pm 1$.
\end{theorem}

\begin{proof}
Each of the sequences $a u_n + b v_n$ corresponds to a prime $\ell < 100$. Suppose that for some $n$, $a u_n + b v_n = x^{11}$ is a perfect $11$-th power. Then, by Lemma \ref{lem:perfectpowersgivecurvepoints}, there is a point of the form $(x, y)$ on one of the hyperelliptic curves
$$ Y^2 = 5 X^{22} \pm 4 \ell. $$
We used a modification of the program from \cite{balakrishnan2020variants}, which implements Barros's algorithm, to obtain Thue equations from the hyperelliptic curves corresponding to the relevant values of $\ell$ (the modified program is available at \cite{DembnerJain2021code}). By explicitly computing the coefficients and discriminants involved in each case, we found that we can take the following bounds:
$$ L = 10^{32}, H = 10^{50}, B = 4. $$
Additionally, the computation showed that for all $\gamma \in \Gamma$, we have $|\gamma| < 10^{40}$. Suppose that $(A,B)$ is a solution to any of the Thue equations under consideration. By Theorem \ref{thm:bounds for thue equation}, it follows that
$$ \max(|A|, |B|) < \exp \{ c_3 R (\log^* R) (R + \log(HB)) \} = \exp \{ c_3 R (\log^* R) (R + 50 \log(10) + 2 \log (2) \} $$
In our case, the extensions all have degree $n = d = 11$. By Dirichlet's unit theorem, it follows that the unit rank $r$ is at most $10$. Therefore, we have
$$ c_3 \leq 3^{37} 11^{89} 11^{22+60+14} = 3^{37} 11^{185}. $$
By Lemma \ref{lem: Boundsonregulators}, we know that $R \leq 2^2 L = 2^2 10^{32}$. Clearly, $50 \log(10) + 2 \log(2) < 2^2 10^{32}$. Therefore, we have 
\begin{align*}
    \max(|A|, |B|) &< \exp \{3^{37} 11^{185} 2^2 10^{32} \log^*(2^2 10^{32}) (2^3) 10^{32}\} \\
    &= \exp \{ \log(2^2 10^{32}) 3^{37} 2^{5} 10^{64} 11^{185} \} \leq \exp(10^{278}).
\end{align*}
By Theorem \ref{thm:barrosthueequations}, we have
$$ y  =  \frac{1}{2} \left( \gamma_{+}(A+B\omega)^{11} + \gamma_-(A+B\overline{\omega})^{11} \right). $$

Our computation showed that in the integral bases $\{ 1, \omega \}$ corresponding to these Thue equations, $\omega$ always has absolute value at most $10$. Therefore, by the triangle inequality, $A + B \omega$ always has absolute value at most $11 \exp(10^{278})$. Furthermore, we have seen that the elements $\gamma_+, \gamma_-$ have absolute value at most $10^{40}$. Therefore, we have
$$ |y| \leq 2 \cdot 10^{40} \cdot 11^{11} \exp(11 \cdot 10^{278}) \leq \exp(10^{280}). $$
Applying that $y^2 = 5 x^{22} \pm 4 \ell$, we have that
$$ |x^{11}| \leq |x^{22}| \leq \frac{1}{5} \left( |y|^2 \mp 4 \ell \right) \leq \exp(10^{281}). $$

It remains to show that $|x^{11}| > \exp(10^{281})$. Since $x^{11} =  x_n := a u_n + b v_n$ is a perfect $11$-th power, it follows from Proposition \ref{prop: 10^300 checking} that $n \geq 10^{300}$. For each of the sequences under consideration, there exists some $k \leq 5$ such that either $x_k \geq u_0 = 0,  x_{k+1} \geq u_1 = 1$, or $x_k \leq 0, x_{k+1} \leq -1$. Without loss of generality, we may assume we are in the first case. Then, since $(x_m)$ satisfies the same recurrence as the Fibonacci sequence $(u_m)$ and its initial conditions starting from $x_k$ dominate the initial conditions $u_0 = 0, u_1 = 1$, we have $|x_{m+k}| \geq |u_m|$ for all $m$. In particular, $|x_{n}| \geq |u_{n-k}|$. However, we know that the Fibonacci sequence is given by
$$ u_m = \frac{1}{\sqrt{5}} \left( \left( \frac{1+\sqrt{5}}{2} \right)^m  - \left(\frac{1-\sqrt{5}}{2} \right)^m \right). $$
As $m \to \infty$, the second term disappears. Therefore, writing $\phi = \frac{1+\sqrt{5}}{2}$, it follows that for sufficiently large $m$, we have
$$ |u_m| \geq  \frac{\phi^m}{\sqrt{5}} - 1. $$
Furthermore, $\phi^3 > e$. It follows that for sufficiently large $m$, we have $|u_{10^m}| > \exp(10^{m-1})$. It follows that
$$ |x_n| \geq |u_{n-k}| \geq |u_{10^{299}}| > \exp(10^{298}). $$
This shows, as claimed, that $|x_n| > \exp(10^{281})$. \qedhere

\end{proof}

It remains to complete the second task by unconditionally solving the relevant Thue equations. The least tractable Thue equations are those of the form
$$ F_{\ell-1}(X,Y) = \pm \ell, $$
for large primes $\ell < 100$. As observed in Lemma 4.1 of \cite{balakrishnan2020variants}, these polynomials are related, via the product formula \eqref{eq:Thueproductformula}, to others studied by Bilu, Hanrot, and Voutier. For any $n$, define the homogeneous polynomial $\widehat{F}_n(X,Y)$ by 
$$\widehat{F}_n(X,Y) = \prod_{k=1}^{\frac{n-1}{2}} \left( Y - 2 X \cos \left( \frac{2 \pi k}{n} \right) \right).$$
Then we have
$$ F_{\ell-1}(X,Y) = \widehat{F}_\ell(X,Y-2X). $$
At this point, we could invoke the main theorem of \cite{bilu2001existence}. Together with Theorem 2.4 of \cite{bilu2001existence}, it implies that for $\ell \geq 31$, $\widehat{F}_\ell$ has no solutions arising from Lucas pairs. For completeness, we instead classify all solutions to $F_{\ell-1}(X,Y) = \pm \ell$, by making use of the following bound also proven in \cite{bilu2001existence}.

\begin{theorem}[\cite{bilu2001existence}, page 105]\label{thm: Bilu bound on Thue solutions}
If $31 \leq p \leq 527$ and $p$ is an odd prime, then the Thue equations
$$ \widehat{F}_p(X,Y) = \pm p $$
have no solutions $(x,y)$ with $|x| > e^8$. 
\end{theorem}

\begin{lemma}\label{lem: solving Thue equations unconditionally}
Let $\ell$ be an odd prime with $\ell < 100$, and suppose that $d \geq 7$ is an odd prime such that $d \mid \ell(\ell^2-1)$. Then all the solutions to $F_{d-1}(X,Y) = \pm \ell$ are listed in Tables 4 and 5 of \cite{balakrishnan2020variants}.
\end{lemma}

\begin{proof}
For odd primes $\ell \leq 37$, this was already proven in \cite{balakrishnan2020variants}. It remains to consider the odd primes $\ell \geq 41$, for which the corresponding Thue equations were only solved conditional on the Generalized Riemann Hypothesis. These equations can be divided into those where $d = \ell$, and those where $d < \ell$. In all the cases where $d < \ell$, the Thue equation solver in \texttt{PARI/GP} confirms that the solutions listed are the only ones. 

In the case where $d = \ell$, the equations are often out of reach of the Thue solver, so we employ the method outlined above. Pick some prime $\ell < 100$, and suppose that $F_{\ell-1}(x,y)  = \widehat{F}_{\ell}(x, y - 2x) = \pm \ell$. By Theorem \ref{thm: Bilu bound on Thue solutions}, it follows that $|x| < e^8 < 3000$. Since $\ell < 100$, a short calculation using the definition of $\widehat{F}_\ell$ shows that $|y - 2x| < 7000$, which implies that $|y| < 13000$. Therefore, in order to show that the only solutions are those listed in the tables of \cite{balakrishnan2020variants}, it is enough to search manually for solutions $(x,y)$ with $|x| < 3000, |y| < 13000$. We implemented an algorithm in \texttt{SageMath}, available at \cite{DembnerJain2021code}, which carried out such a search for the primes $\ell$ between $41$ and $100$. By doing so, we confirmed that the only solutions are those listed. \qedhere

\end{proof}

\begin{remark}
This result also allows many of the results proven conditionally in Theorem 1.3 of \cite{balakrishnan2020variants} to be verified unconditionally.
\end{remark}

Now, we are in a position to prove our main result for this section.

\begin{proof}[Proof of Theorem \ref{Maintheorem1}]
Suppose that $\ell$ is an odd prime less than $100$, and that $\tau(n) = \pm \ell$ for some $\ell$. By Theorem \ref{thm:divisibilityfortau}, it follows that $n = p^{d-1}$, where $p$ and $d$ are odd primes, and $d \mid \ell(\ell^2-1)$. Now, there are three cases to consider.

First, suppose $d = 3$. In this case, we obtain a solution to $Y^2 = X^{11} \pm \ell$. Previous work on these equations, summarized in Tables 6 and 7 of \cite{balakrishnan2020variants}, shows that none of the points on these curves correspond to primes $p$ for which $\tau(p^2) = \pm \ell$.

Second, suppose $d = 5$. In this case, we obtain a solution to $Y^2 = 5 X^{22} \pm 4 \ell$. By Lemma \ref{prop:pellsolutionsgiveperfect powers}, this gives rise to a perfect power in one of several recurrence sequences. Taken together, Theorems \ref{thm:rulingoutperfectpowersbycongruences} and \ref{thm: Finishing the hard sequences} show that the sequences corresponding to primes $\ell \equiv \pm 1 \pmod 5$ contain no perfect powers aside from $\pm 1$. 

Finally, suppose $d \geq 7$. In this case, we obtain a solution $(x,y)$ to a Thue equation $F_{d-1}(X,Y) = \pm \ell$, where $x = p^{11}$. By Lemma \ref{lem: solving Thue equations unconditionally}, none of the relevant Thue equations has a solution $(x,y)$ where $x$ is of the form $p^{11}$. \qedhere

\end{proof}

\begin{remark}
Theorem \ref{Maintheorem1} has two natural extensions. First, one could replace the tau-function with $a_f(n)$ for a different newform $f$ with integer coefficients. Here, the approach carries through in much the same way. The results for Thue equations and the curves $C_{d, \ell}$ apply verbatim in this case. The results for the curves $H_{2k-1, \pm \ell}: Y^2 = 5 X^{2(2k-1)} \pm 4 \ell$ extend in exactly the same way to values $2k \neq 12$. However, the perfect powers involved would no longer be perfect $(2k-1)$-st powers, so a new sieve computation and a new upper bound would be required.

Second, one could replace the odd prime $\ell$ by an odd composite value $\alpha< 100$. Here, there are more difficulties. Once again, known results on the curves $C_{d, \alpha}$ would apply, and once again, a similar computation can classify points on the curves $H_{2k-1, \pm \alpha}$, although there will be more sequences to check. However, Theorem \ref{thm: Bilu bound on Thue solutions} does not apply to Thue equations of the form $F_{p-1}(X,Y) = \pm \alpha$ where $p$ divides $\alpha$, which means the argument would likely depend on Thue equations that are difficult to solve.

\end{remark}

\section{Modular Approaches to $A X^p + B Y^q = C Z^r$} \label{sec:modularapproaches}

In this section, we describe a modern method for solving certain Diophantine equations. Later, we will see that newform coefficients are related to certain solutions to equations
\begin{equation} \label{generalizedfermat}
A X^p + B Y^q = C Z^r. 
\end{equation}
Equations of this type can often be successfully solved by using a modular method analogous to the proof of Fermat's Last Theorem. Here, we summarize this strategy, and in the next section we apply it to prove several results about newform coefficients.

Let $E$ be an elliptic curve over $\bbQ$. For all primes $p$, the curve $E$ has an associated mod $p$ Galois representation
$$ \rho_{E, p} \colon G_{\bbQ} \to \GL_2(\bbF_p), $$
corresponding to the action of $G_{\bbQ}$ on the $p$-torsion $E[p]$. The modularity theorem says that the representations $\rho_{E,p}$ are related to representations $\rho_{f,p}$ coming from newforms $f$ with integer coefficients.

\begin{theorem}[Modularity of elliptic curves \cite{Wiles1995modular}, \cite{taylor1995heckealgebras}, \cite{breuil2001modularity}] \label{thm:modularity}
Let $E$ be an elliptic curve over $\bbQ$ with conductor $N$. For any prime $p$, there exists a newform $f$ of weight $2$ and level $N$ with integer coefficients, such that the representations $\rho_{E,p}$ and $\rho_{f,p}$ are isomorphic over $\overline{\bbF}_p$. Conversely, given any such newform $f$ with integer coefficients, there is an elliptic curve $E$ such that $\rho_{E,p}$ and $\rho_{f, p}$ are isomorphic.
\end{theorem}

Given some $A,B,C$, if $(x,y,z)$ is an integer solution to \eqref{generalizedfermat}, one can associate an elliptic curve $E / \bbQ$, called a \emph{Frey curve}, to the solution $(x,y,z)$. By the modularity theorem, this corresponds to a newform $f$. If the possibility of such a newform $f$ can be ruled out, for instance because the newspace of level $N$ is trivial, then this leads to a contradiction. A key ingredient in applying the modular method is the following theorem of Ribet, which allows us to reduce the level of the form $f$ under consideration.

\begin{theorem}[Level Lowering \cite{ribet1991levellowering}]\label{thm:level lowering}

Let $f$ be a weight $2$ newform of level $\ell N$ where $\ell$ does not divide $N$, and let $K_f$ be the number field generated by its coefficients. Suppose that $p$ is a rational prime and $\mathfrak{p}$ is a prime of $\overline{\bbQ}$ lying over $p$.  Suppose that $\rho_{f,\mathfrak{p}}$ is absolutely irreducible, and that either:
\begin{itemize}
    \item $\rho_{f,\mathfrak{p}}$ is unramified at $\ell$, or
    \item $\ell = p$ and $\rho_{f, \mathfrak{p}}$ is flat at $p$. 
\end{itemize}
Then there is a weight two newform $g$ of conductor $N$ such that $\rho_{f, \mathfrak{p}}$ is isomorphic to $\rho_{g, \mathfrak{p}}$ over $\overline{\bbF}_p$.

\end{theorem}

Many applications of the modular method can be improved conditional on the Frey-Mazur Conjecture, which says that for most primes $p$, elliptic curves over $\bbQ$ are determined up to isogeny by their mod $p$ Galois representation.

\begin{conjecture}[Frey-Mazur]\label{conj:Frey-Mazur}
Let $E_1, E_2 / \bbQ$ be elliptic curves. If the representations $\rho_{E_1, p}$ and $\rho_{E_2, p}$ are isomorphic for some $p > 17$, then $E_1$ and $E_2$ are isogenous over $\bbQ$. 
\end{conjecture}

The importance of the Frey-Mazur Conjecture comes from the fact that it rules out the possibility of nontrivial level-lowering between rational forms when $p > 17$. If $E / \bbQ$ has conductor $N$  and $\rho_{E,p}$ arises from some newform $g$ of level $M < N$, then the form $g$ may or may not have integer (equivalently, rational) coefficients. If $g$ does not have integer coefficients, then the Frey-Mazur Conjecture offers no help. However, if $g$ has integer coefficients, then we have the following.

\begin{proposition}
Assume the Frey-Mazur Conjecture. If $E / \bbQ$ has conductor $N$ and $p > 17$, then there is no newform $g$ with integer coefficients and level $M < N$ such that $\rho_{E, p} \cong \rho_{g, p}$. 
\end{proposition}

\begin{proof}
Given some such newform $g$, the modularity theorem guarantees that $\rho_{g, p} \cong \rho_{E', p}$ for some elliptic curve $E' / \bbQ$ of conductor $M$. It follows that $\rho_{E, p} \cong \rho_{E',p}$, so by the Frey-Mazur Conjecture, $E$ and $E'$ are isogenous. But this implies that $E$ and $E'$ have the same conductor, yielding a contradiction. \qedhere
\end{proof}

\section{Applications of Modularity}\label{sec: Applications of modularity}

We conclude by proving several more uniform results, applying to all newforms of sufficiently large weight with integer coefficients. The results in this section rely on the modular method, and many are conditional on the Frey-Mazur Conjecture.

The key observation in applying the modular method to newform coefficients is that for any value $\alpha$, if $a_f(p^2) = \alpha$ or $a_f(p^4) = \alpha$, then the equations arising from Corollary \ref{cor: recurrencesgivediophantineequations} can be interpreted as Fermat equations of the form
\begin{equation}\label{eq:General(n,n,2)equation}
AX^n + B Y^n = C Z^2.
\end{equation}
In general, equations of this form will have many integer solutions. For example, suppose that $n$ is odd and that $C = 1$. Then for any $a,b$, the triple $(ac, bc, c^{\frac{n+1}{2}})$ is a solution, where $c = A a^n + B b^n$. This naturally leads us to restrict the set of solutions under consideration.

\begin{definition}
Let $(a,b,c)$ be an integer solution to \eqref{eq:General(n,n,2)equation}. We say that $(a,b,c)$ is a \emph{primitive} solution if the elements $Aa, Bb,$ and $Cc$ are pairwise coprime.
\end{definition}

\begin{proposition}\label{prop:CandHFermatEquations}
Let $f$ be a newform with integer coefficients, even integer weight $2k$, and level $N$. Then the following are true.
\begin{enumerate}
    \item Suppose that $p \nmid N$ is a prime and that $a_f(p^2) = \alpha$. If $\ell \mid 2k-1$, then $(X,Y,Z) = (p^{(2k-1) / \ell}, 1, a_f(p))$ is a solution to the twisted Fermat equation
\begin{equation}\label{eq:CFermatEquation}
    X^{2k-1} + \alpha Y^{2k-1} = Z^2.
\end{equation}
    If $p$ does not divide $\alpha$, then the solution is primitive.
    \item If $p \nmid N$ is prime and $a_f(p^4) = \alpha$, then $(X,Y,Z) = (p^{2(2k-1) / \ell}, 1, 2 a_f(p)^2 - 3 p^{2k-1})$ is a solution to the twisted Fermat equation
\begin{equation}\label{eq:HFermatequation}
    5 X^{2k-1} + 4 \alpha Y^{2k-1} = Z^2.
\end{equation}
    If $p$ does not divide $2 \alpha$ and $5$ does not divide $\alpha$, then the solution is primitive.
\end{enumerate}
\end{proposition}

\begin{proof}
It follows immediately from Corollary \ref{cor: recurrencesgivediophantineequations} that the given equations are satisfied, so we only need to show that the solutions are primitive. Observe that in order to show the solution $(a,b,c)$ is primitive, it's enough to show that $\gcd(Aa, Bb, Cc) = 1$. Indeed, if any prime divides two of the factors, then it divides two of the terms in \eqref{eq:General(n,n,2)equation}, and therefore divides the third.

First, consider case (1). In this case, we need to show that $\gcd(p, a_f(p), \alpha) = 1$. Since $p$ does not divide $\alpha$, this is immediate from what we wrote above.

Now, consider case (2). In this case, we need to show that $\gcd(5p, 4 \alpha, 2a_f(p)^2 - 3 p^{2k-1}) = 1$. Since $p$ does not divide $2 \alpha$ and $5$ does not divide $\alpha$, it follows that the first two terms are coprime, which once again completes the argument. \qedhere
\end{proof}

In the paper \cite{bennett2004ternary}, Bennett and Skinner associate a Frey curve to any primitive solution of \eqref{eq:General(n,n,2)equation}. These Frey curves will form the basis for our work in this section.

\begin{lemma}[\cite{bennett2004ternary}, Lemma 2.1]\label{lem:Freycurves}
Let $A,B,C$ be any nonzero integers, where $C$ is squarefree. Suppose that $(a,b,c)$ is a primitive solution to \eqref{eq:General(n,n,2)equation}, where $n$ is an odd prime. Then there is an associated elliptic curve $E = E(a,b,c)$, with the conductor $N(E)$ given by 
    $$ N(E) = 2^{\alpha} C^2 \prod_{p \mid ab AB} p, $$
    where $\alpha \in \{-1,0,\ldots, 6 \}$. If $4 \mid B$, then $\alpha \leq 4$. If $\ord_2(B) = 2$, then we have
    $$ \alpha = \begin{cases} 
    1 & b \equiv -BC / 4 \pmod 4, \\
    2 & b \equiv BC  / 4 \pmod 4.
    \end{cases}$$
\end{lemma}

\begin{lemma}[\cite{bennett2004ternary}, Lemmas 3.2 and 3.3]\label{lem:artinconductor}
Suppose that $n \geq 7$ is a prime and that $\rho_{E,n}$ is associated to a primitive solution $(a, b, c)$ to \ref{eq:General(n,n,2)equation} with $ab \neq \pm 1$. If $n \nmid ABC$, then define
$$N_n(E) := 2^\beta \prod_{p|C, p \nmid n} p^2 \prod_{q|AB, q \neq n} q,$$
where 
$$\beta = \begin{cases} 1 & ab \equiv 0 \pmod{2} \text{ and } AB \equiv 1 \pmod{2} \\ \alpha & \text{ otherwise.}\end{cases}$$
The representation $\rho_{E,n}$ arises from a cuspidal newform of weight $2$, level $N_n(E)$, and trivial nebentypus character. 
\end{lemma}

As we noted above, the Frey-Mazur Conjecture only rules out level-lowering between rational newforms. To deduce consequences for specific Diophantine equations, we also need to rule out the possibility of level-lowering from a rational form to an irrational form. In certain cases, we can do this directly. In other cases, we can apply the following (unconditional) result, which tells us that such level-lowering cannot happen whenever the levels in question are sufficiently large. Let $\psi$ denote the Dedekind $\psi$-function:
$$\psi(N) = N \prod_{p \mid N} \left( 1 + \frac{1}{p} \right).$$
\begin{theorem}[\cite{silliman2015powers}, Theorem 4.2]\label{thm:boundsforirrationallevel-lowering}
Let $f$ be an irrational newform of conductor $N$ with coefficients in a number field $K_f / \bbQ$. Suppose that for some prime $\mathfrak{p} / p$ of $K_f$ and some elliptic curve $E / \bbQ$, the representations $\rho_{E, p}$ and $\rho_{f, \mathfrak{p}}$ are isomorphic. Then
$$ p \leq \psi(N)^{1 + \psi(N) / 12}. $$
\end{theorem}

We define the following two functions:
$$f_C(m) = \max \left(17, \psi\left(2^6 \prod_{q \mid m} q \right)^{1+\psi(2^6 \prod_{q \mid m} q) / 12}\right), $$
and
$$ f_H(m) = \max\left(17, \psi\left(2^4 \prod_{q \mid 5 m} q\right)^{1 +\psi(2^4 \prod_{q \mid 5 m} q) / 12}\right).  $$

\begin{theorem}\label{thm: Neglected modularity result}
Let $m$ be any nonzero integer, and let $f$ be a newform with integer coefficients and weight $2k$. Assuming the Frey-Mazur Conjecture, we have:

\begin{enumerate}
    \item If $2k-1$ is divisible by a prime $\ell > f_C(m)$ and $p$ is a prime not dividing $m$, then $a_f(p^2) \neq m$.
    \item If $2k-1$ is divisible by a prime $\ell > f_H(m)$, $p$ is a prime not dividing $2m$, and $m$ is coprime to $5$, then $a_f(p^4) \neq m$.
\end{enumerate}
\end{theorem}

\begin{proof}
We will prove (1); the proof of (2) is similar. Suppose that the given conditions hold. By Theorem \ref{thm:propertiesofnewforms}, we may assume that $p \nmid N$. By Proposition \ref{prop:CandHFermatEquations}, we obtain a primitive solution $(p^{(2k-1) / \ell}, 1, a_f(p))$ of the equation
$$ X^{\ell} + m Y^{\ell} = Z^2. $$
By Lemma \ref{lem:Freycurves}, this gives rise to a Frey curve $E$ with conductor
$$ N(E) = 2^{\alpha} \prod_{q \mid p m} q. $$
By Lemma \ref{lem:artinconductor}, the associated representation arises from a newform $f$ of level
$$ 2^{\alpha} \prod_{q \mid m} q. $$
Since $\ell > f_C(m)$, it follows from Theorem \ref{thm:boundsforirrationallevel-lowering} along with the formula for the level of the representation that $f$ has integer coefficients. It follows by the Frey-Mazur Conjecture that 
$$N(E) = 2^{\alpha} \prod_{q \mid m } q.$$
However, this would imply that $p \mid m$, contradicting our initial assumption. \qedhere

\end{proof}

\begin{theorem}\label{thm:First modular result}
Let $f$ be a newform with integer coefficients. Then the following are true.

\begin{enumerate}
    \item If $f$ has weight $2k$ and $p$ is an ordinary prime for $f(z)$, i.e., $p \nmid a_f(p)$, then $a_f(p^2) \neq m^j$ for any $j \geq 4$ dividing $2k-1$ and any nonzero integer $m$. In particular, if $2k \geq 6$, then $a_f(p^2) \neq m^{2k-1}$.
    \item Assume the Frey-Mazur Conjecture. If $f$ has weight $2k$, $2k-1$ is divisible by a prime $\ell \geq 19 $, and $p \neq 2,5$ is ordinary, then $a_f(p^4) \neq m^\ell$ for any nonzero integer $m$.
\end{enumerate}

\end{theorem}

\begin{proof}[Proof of Theorem \ref{thm:First modular result} and Theorem \ref{Maintheorem3} parts (1) and (2)]
First, we prove statement (1). Suppose that the given conditions hold, and that $j \geq 4$ divides $2k-1$. We may again assume that $p \nmid N$. Since $p$ is ordinary, $p$ is coprime to $m$, so we obtain a primitive solution $(p^{(2k-1) / j}, m, a_f(p))$ to the Fermat equation
$$X^j +Y^j = Z^2. $$
By work of Darmon and Merel \cite{Darmon1997winding} and Poonen \cite{Poonen1998some}, this has no primitive solutions when $j \geq 4$.

Now, we prove statement (2). First, assume  that $m$ is coprime to $5$. In this case, we obtain a primitive solution $(p^{(2k-1) / \ell}, m^{(2k-1) / \ell}, 2 a_f(p)^2 - 3 p^{2k-1})$ to the Fermat equation
\begin{gather*}
    5 X^{\ell} + 4 Y^{\ell} = Z^2.
\end{gather*}
By Lemma \ref{lem:Freycurves}, this gives rise to a Frey curve $E$ with conductor
$$ N(E) = 2^{\alpha} \prod_{q \mid 20 p m} q. $$
By Lemma \ref{lem:artinconductor}, the associated representation arises from a newform $f$ of level
$$ 2^{\alpha} \prod_{q \mid 20} q. $$
In our case, we have $\alpha = 1$ or $\alpha = 2$. Therefore, this expression equals either $20$ or $40$. At both levels, there is only one newform with weight $2$ and trivial nebentypus, and both have integer coefficients. By the Frey-Mazur Conjecture, we have
$$ N(E) = 2^{\alpha} \prod_{q \mid 20} q. $$
It follows that $p = 2,5$. However, this contradicts our initial assumption.

If $m$ is not coprime to $5$, then it follows that $5$ divides the lefthand side of the equation
\begin{gather}\label{Key equation}
(2 a_f(p)^2 - 3 p^{2k-1})^2 = 5 p^{2(2k-1)} + 4 m^{\ell}.
\end{gather}
Since the lefthand side is a square, it follows that $25$ divides the righthand side of \ref{Key equation}. Therefore, $5 \mid m^{\ell}$. Since $\ell > 1$, $25$ divides $m^{\ell}$. It follows that $25$ divides $5 p^{2(2k-1)}$ as well. Given that $p$ must be prime, we have $p = 5$, a contradiction. \qedhere

\end{proof}

Even in cases where some irrational newforms exist at the given level, one can often use congruence information from a given irrational newform to show that it gives rise to no primitive solutions.

\begin{proposition}[\cite{bennett2004ternary}, Proposition 4.3]
Let $n \geq 7$ be a prime, and let $E = E(a,b,c)$ be a Frey curve associated to a primitive solution of \eqref{eq:General(n,n,2)equation}. If the representation $\rho_{E,p}$ arises from a newform $f$ with coefficients in a number field $K_f$, if $p$ is a prime coprime to $n$, and $E$ has good reduction at $p$, then $n$ divides
$$ \operatorname{Norm}_{K_f / \bbQ} (c_p \pm 2r), $$
for some $r \leq \sqrt{p}$.
\end{proposition}

\begin{remark}
In the case where $K_f = \bbQ$, this gives little information. Indeed, by the Hasse bound, we will always have $|c_p| \leq 2 \sqrt{p}$, so if $c_p$ is even then this condition is always satisfied.
\end{remark}

By employing arguments of this type, we obtain the following generalization of Theorem \ref{Maintheorem1} in the case $\ell = 19$.

\begin{theorem}\label{thm:Second modular result}
Assume the Frey-Mazur Conjecture. Let $f$ be a newform with integer coefficients, even integer weight $2k$, residually reducible mod $2$ Galois representation, and level $N$ coprime to $19$. Suppose that $2k-1$ is divisible by a prime $\ell > 19$. If $a_f(5^4), a_f(19^4) \neq \pm 19$, then $a_f(n) \neq \pm 19$ for any $n$.
\end{theorem}

\begin{proof}[Proof of Theorem \ref{thm:Second modular result} and Theorem \ref{Maintheorem3} part (3)]
By Theorem \ref{thm:divisibilityfortau}, if $a_f(n) = 19$, then $n = p^{d-1}$, where $p$ and $d$ are odd primes, $p \nmid N$, and $d \mid 19(19^2- 1)$. It follows that $d = 3,5,$ or $19$. By Corollary \ref{cor: recurrencesgivediophantineequations}, these three cases each correspond to solutions of Diophantine equations. In the case $d = 3$, we obtain a solution to
$$ x^2 \mp 19 = y^{j}, $$
where $j = 2k-1$. According to the tables in \cite{bugeaud2006classical2} and \cite{barros2010lebesgue}, this equation has no solutions except for $j = 2,3,5$. Since $j = 2k-1$ and $\ell > 19$ divides $2k-1$, this shows that none of these solutions correspond to coefficients $a_f(p^2)$.

In the case $d = 19$, we obtain a Thue equation, and the tables computed in \cite{balakrishnan2020variants} indicate all its solutions are of the form $(1, \pm 4)$. Since $1 \neq p^{11}$ for any prime $p$, no such solution comes from a prime $p$ with $a_f(p^4) = 19$.

In the case $d = 5$, we obtain a solution $(p^{2(2k-1) / \ell}, 1, 2 a_f(p)^2 - 3 p^{2k-1})$ to the Fermat equation
$$5 X^{\ell} \pm 76 Y^{\ell} = Z^2. $$
Since we are excluding the cases $p = 5$ and $p = 19$, this solution is primitive. By Lemma \ref{lem:Freycurves}, such a solution corresponds to an elliptic curve $E$ with conductor $ N(E) = 2^{2} \cdot 5 \cdot 19 \cdot p.$ Since $\ell > 19$, $\ell$ is coprime to $76$. Therefore, by Lemma \ref{lem:artinconductor}, the representation $\rho_{E, 19}$ is associated to a newform $f$ of level $2^2 \cdot 5 \cdot 19 = 380$. There are four newforms of weight $2$ and level $380$ with trivial nebentypus character, two of which are rational. First, suppose that $f$ is rational. Then by the Frey-Mazur Conjecture, we must have $N(E) = 2^2  \cdot 5  \cdot  19$. But this implies that $ p = 1$, which is a contradiction.

It remains to show that no solutions can come from either of the two irrational newforms. The first such newform has LMFDB label 380.2.a.c, and its coefficients are defined over $\bbQ(\sqrt{2})$. In this case, we have $c_3 = -2 + 2 \sqrt{2}$. The elements $-2 + 2 \sqrt{2} \pm 2r$ for $r = 0,\pm 1$ have norms in $ \{2, -2, 14\}$. Since none of these norms are divisible by the prime $\ell \geq 19$, it follows that there are no primitive solutions arising from this newform.

The other irrational newform of level $380$ has LMFDB label 380.2.a.d. In this case, we have $c_3 = 1 + \sqrt{3}$. For $r = 0, 1, -1$, the elements $c_3 \pm 2r$ have norms in $ \{ -2, 6 \}$. As above, we conclude that there are no primitive solutions arising from this newform, and this completes the proof. \qedhere

\end{proof}

\begin{remark}
A natural question is whether this theorem can be made unconditional by explicitly ruling out the two rational newforms, which have labels 380.2.a.b and 380.2.a.a. We can rule out the newform 380.2.a.b by considering the coefficient $c_3 = 2$. If a solution $(a,b,c)$ came from the rational newform $f$, then by the setup in Bennett and Skinner, it would follow that the Frey curve
$$E: Y^2 =  X^3 + c X^2 + 19X $$
either has bad reduction at $3$, or satisfies $3+1 - \#E(\bbF_3) = 2$. The case of bad reduction would correspond to the possibility $a_f(3^4) = \pm 19$, which can be ruled out by hand, and the condition on $\#E(\bbF_3)$ is not satisfied by any $\overline{c} \in \bbF_3$

The newform 380.2.a.a cannot be ruled out by similar means, because it corresponds to the trivial solutions $(1, \pm 9)$ of the equation $Y^2 = 5 X^{22} + 76$. In certain cases, one can use the modular method to constrain the solutions of an equation that has some solutions, as in Section 9 of \cite{bugeaud2006classical2}, but such approaches seem not to help in this case.
\end{remark}

\printbibliography

\Addresses

\end{document}